\documentclass{amsart}
\usepackage{amssymb,amscd,amsfonts,amsbsy,enumerate}
\usepackage{latexsym}
\usepackage{exscale}
\usepackage{amsmath,amsthm,amsfonts}
\usepackage{mathrsfs}

\usepackage[colorlinks=true, pdfstartview=FitV, linkcolor=blue, citecolor=red, urlcolor=blue]{hyperref}

\usepackage{color}

\numberwithin{equation}{section}

\allowdisplaybreaks

\newtheorem{theorem}{Theorem}[section]
\newtheorem{lemma}[theorem]{Lemma}

\newtheorem{proposition}[theorem]{Proposition}
\newtheorem{definition}[theorem]{Definition}
\newtheorem{convention}[theorem]{Convention}

\theoremstyle{remark}
\newtheorem{remark}[theorem]{Remark}

\begin{document}
\allowdisplaybreaks

\title[The higher order Dirichlet problem in the upper-half space]
{The higher order regularity Dirichlet problem for
elliptic systems in the upper-half space}

\author{Jos\'e Mar{\'\i}a Martell}
\address{Jos\'e Mar{\'\i}a Martell
\\
Instituto de Ciencias Matem\'aticas CSIC-UAM-UC3M-UCM
\\
Consejo Superior de Investigaciones Cient{\'\i}ficas
\\
C/ Nicol\'as Cabrera, 13-15
\\
E-28049 Madrid, Spain} \email{chema.martell@icmat.es}

\author{Dorina Mitrea}
\address{Dorina Mitrea
\\
Department of Mathematics
\\
University of Missouri
\\
Columbia, MO 65211, USA} \email{mitread@missouri.edu}

\author{Irina Mitrea}
\address{Irina Mitrea
\\
Department of Mathematics
\\
Temple University\!
\\
1805\,N.\,Broad\,Street
\\
Philadelphia, PA 19122, USA} \email{imitrea@temple.edu}

\author{Marius Mitrea}
\address{Marius Mitrea
\\
Department of Mathematics
\\
University of Missouri
\\
Columbia, MO 65211, USA} \email{mitream@missouri.edu}

\thanks{The first author has been supported in part by MINECO Grant MTM2010-16518 and ICMAT Severo Ochoa project SEV-2011-0087,
the second author has been supported in part by a Simons Foundation grant 200750
and by a University of Missouri research leave, the third author has been supported
in part by US NSF grant 0547944. The fourth author has been supported in part by the Simons Foundation grant 281566.
This work has been possible thanks to the support and hospitality
of \textit{Temple University} (USA), \textit{ICMAT, Consejo Superior de Investigaciones Cient{\'\i}ficas} (Spain) and the \textit{Universidad Aut\'onoma de Madrid} (Spain). The authors express their gratitude to these institutions.}

\date{October 30, 2012. \textit{Revised}: \today}



\begin{abstract}
We identify a large class of constant (complex) coefficient, second order elliptic
systems for which the Dirichlet problem in the upper-half space with data in
$L^p$-based Sobolev spaces, $1<p<\infty$, of arbitrary smoothness $\ell$, is well-posed
in the class of functions whose nontangential maximal operator of their derivatives
up to, and including, order $\ell$ is $L^p$-integrable. This class includes all
scalar, complex coefficient elliptic operators of second order, as well as the
Lam\'e system of elasticity, among others.
\end{abstract}

\maketitle

%

\allowdisplaybreaks

\section{Introduction}
\setcounter{equation}{0}
\label{S-1}

Let $M$ be a fixed positive integer and consider the second-order, $M\times M$ system,
with constant complex coefficients, written as
\begin{equation}\label{L-def}
Lu:=\Bigl(\partial_r(a^{\alpha\beta}_{rs}\partial_s u_\beta)
\Bigr)_{1\leq\alpha\leq M}
\end{equation}
when acting on a ${\mathscr{C}}^2$ vector valued function $u=(u_\beta)_{1\leq\beta\leq M}$.
A standing assumption for this paper is that $L$ is {\tt elliptic}, in the sense that
there exists a real number $\kappa_o>0$ such that the following Legendre-Hadamard
condition is satisfied (here and elsewhere, the usual convention of summation
over repeated indices is used)
\begin{equation}\label{L-ell.X}
\begin{array}{c}
{\rm Re}\,\bigl[a^{\alpha\beta}_{rs}\xi_r\xi_s\overline{\eta_\alpha}
\eta_\beta\,\bigr]\geq\kappa_o|\xi|^2|\eta|^2\,\,\mbox{ for every}
\\[8pt]
\xi=(\xi_r)_{1\leq r\leq n}\in{\mathbb{R}}^n\,\,\mbox{ and }\,\,
\eta=(\eta_\alpha)_{1\leq\alpha\leq M}\in{\mathbb{C}}^M.
\end{array}
\end{equation}

The $L^p$-Dirichlet boundary problem associated with the operator $L$ in the upper-half
space is formulated as $Lu=0$ in $\mathbb{R}^{n}_+$,
$u\bigl|_{\partial\mathbb{R}^{n}_{+}}^{{}^{n.t.}}=f\in L^p(\mathbb{R}^{n-1})$,
and ${\mathcal{N}}u\in L^p(\partial\mathbb{R}^{n}_+)$. Here and elsewhere,
$\mathcal{N}$ denotes the nontangential maximal operator,
while $u\bigl|_{\partial\mathbb{R}^{n}_{+}}^{{}^{n.t.}}$ stands for
the non-tangential trace of $u$ onto $\partial{\mathbb{R}}^n_{+}$ (for precise
definitions see \eqref{NT-1} and \eqref{nkc-EE-2}). While in the particular case $L=\Delta$,
the Laplacian in ${\mathbb{R}}^n$, this boundary value problem has been treated at length in
many monographs, including \cite{GCRF85}, \cite{St70}, \cite{Stein93}, to give just a few examples,
much remains to be done.

Here we are interested in identifying a class of elliptic systems $L$ for which the Dirichlet
problem in the upper-half space is well-posed for boundary data belonging to
higher-order smoothness spaces, such as $L^p_\ell({\mathbb{R}}^{n-1})$, the $L^p$-based Sobolev
space in ${\mathbb{R}}^{n-1}$ of order $\ell\in{\mathbb{N}}_0$, with $p\in(1,\infty)$. In such a
scenario, we shall demand that one retains nontangential control of higher-order derivatives
of the solution. More precisely, given any $\ell\in{\mathbb{N}}_0$, we formulate the $\ell$-th
order Dirichlet boundary value problem for $L$ in $\mathbb{R}^{n}_+$ as follows
\begin{equation}\label{Dir-BVP-p:l}
\left\{
\begin{array}{l}
Lu=0\,\,\mbox{ in }\,\,\mathbb{R}^{n}_+\,\,\mbox{ and }\,\,
u\bigl|_{\partial\mathbb{R}^{n}_{+}}^{{}^{n.t.}}=f\in L^p_\ell(\mathbb{R}^{n-1}),
\\[8pt]
{\mathcal{N}}(\nabla^k u)\in L^p(\partial\mathbb{R}^{n}_+)
\,\,\mbox{ for }\,\,k\in\{0,1,...,\ell\},
\end{array}
\right.
\end{equation}
where $\nabla^k u$ denotes the vector with components
$(\partial^\alpha u)_{|\alpha|=k}$. No concrete case of \eqref{Dir-BVP-p:l} has
been dealt with for arbitrary values of the smoothness parameter $\ell$, so
considering even $L=\Delta$ in such a setting is new. In fact, we are able to
treat differential operators that are much more general than the Laplacian,
again, in the context when the boundary data exhibit an arbitrary amount of
regularity, measured on the $L^p$-based Sobolev scale.

In dealing with \eqref{Dir-BVP-p:l}, the starting point is the fact that, as
known from the seminal work of S.\,Agmon, A.\,Douglis, and L.\,Nirenberg in
\cite{ADNI} and \cite{ADNII}, every constant coefficient elliptic operator $L$ has
a Poisson kernel $P^L$, an object whose properties mirror the most basic characteristics
of the classical harmonic Poisson kernel
\begin{eqnarray}\label{Uah-TTT}
P^{\Delta}(x'):=\frac{2}{\omega_{n-1}}\frac{1}{\big(1+|x'|^2\big)^{\frac{n}{2}}}
\qquad\forall\,x'\in{\mathbb{R}}^{n-1},
\end{eqnarray}
where $\omega_{n-1}$ is the area of the unit sphere $S^{n-1}$ in ${\mathbb{R}}^n$.
In particular, using the notation $F_t(x'):=t^{1-n}F(x'/t)$ for each $t>0$ where $F$
is a generic function defined in ${\mathbb{R}}^{n-1}$, we have
\begin{equation}\label{tghn-jan-1}
|P^L_t(x')|\leq C\frac{t}{(t^2+|x'|^2)^{\frac{n}{2}}}
\qquad\forall\,x'\in{\mathbb{R}}^{n-1},\quad\forall\,t>0.
\end{equation}
Then, given any $f\in L^p({\mathbb{R}}^{n-1})$, $1<p<\infty$, if ${\mathcal{M}}$ stands
for the Hardy-Littlewood maximal operator in ${\mathbb{R}}^{n-1}$, the function
\begin{equation}\label{tghn-jan-2}
u(x',t):=\big(P^L_t\ast f\big)(x'),\qquad\forall\,(x',t)\in{\mathbb{R}}^n_{+},
\end{equation}
satisfies $Lu=0$ in ${\mathbb{R}}^n_{+}$ as well as
$u\bigl|_{\partial\mathbb{R}^{n}_{+}}^{{}^{n.t.}}=f$ a.e. in $\mathbb{R}^{n-1}$, and
\begin{equation}\label{tghn-jan-3}
\big({\mathcal{N}}u\big)(x')\leq C\big({\mathcal{M}}f\big)(x'),
\qquad\forall\,x'\in{\mathbb{R}}^{n-1}.
\end{equation}
In turn, the pointwise estimate \eqref{tghn-jan-3} and the boundedness of
${\mathcal{M}}$ on $L^p({\mathbb{R}}^{n-1})$, $1<p<\infty$, can be used to show,
much as in the case for the Laplacian, that $u$ from \eqref{tghn-jan-2} solves
the $L^p$-Dirichlet problem in the upper half-space for any given constant
coefficient elliptic operator $L$. This corresponds to the case $\ell=0$
in \eqref{Dir-BVP-p:l}.

This being said, it is unclear whether the Agmon-Douglis-Nirenberg Poisson kernel
for a generic elliptic operator $L$ continues to work just as well in the setting
when the boundary data is assumed to have higher order regularity. The issue is that,
in this scenario, one is required to estimate the size of the nontangential maximal
operator of iterated gradients of the solution. For such a goal, in order
to make use of the higher order regularity assumption on the boundary data,
one necessarily must find a way of passing generic derivatives inside the
convolution \eqref{tghn-jan-2}, while at the same time allowing kernels,
of an auxiliary nature, to take the role of the original Poisson kernel.
The caveat is that the nontangential maximal function of convolutions with
these auxiliary kernels should have appropriate control, a matter which may
not always be ensured.

To better understand the nature of this difficulty, consider the case of
\eqref{Dir-BVP-p:l} with $\ell=1$, a scenario in which one still looks for
a solution as in \eqref{tghn-jan-2} (keeping in mind that now $f$ belongs to
the Sobolev space $L^p_1({\mathbb{R}}^{n-1})$, $1<p<\infty$). As far as estimating
${\mathcal{N}}\big(\partial_{x_j}u\big)$ is concerned, it is clear from
\eqref{tghn-jan-2} that only the derivative in the normal direction
(i.e., for $\partial_t\equiv\partial_{x_n}$) is potentially problematic.
In the absence of additional information about the nature of the Poisson kernel
$P^L$ one tool that naturally presents itself is a general identity, valid for
any function $F\in{\mathscr{C}}^1({\mathbb{R}}^{n-1})$, to the effect that
\begin{equation}\label{ioaTga-444}
\partial_t\big[F_t(x')\big]=-\sum_{j=1}^{n-1}
\partial_{x_j}\Big[\big(x_jF(x')\big)_t\Big]
\,\,\mbox{ for every }\,\,(x',t)\in{\mathbb{R}}^n_{+}.
\end{equation}
For $u$ as in \eqref{tghn-jan-2}, this permits us to express
\begin{align}\label{ioaTga-555}
\partial_t\big[u(x',t)\big] &=\partial_t\big[\big(P^L_t\ast f\big)(x')\big]
=-\sum_{j=1}^{n-1}\partial_{x_j}\Big[\big(R^{(j)}_t\ast f\big)(x')\Big]
\nonumber\\[0pt]
&=-\sum_{j=1}^{n-1}\Big[R^{(j)}_t\ast\big(\partial_jf\big)\Big](x')
\,\,\mbox{ for every }\,\,(x',t)\in{\mathbb{R}}^n_{+},
\end{align}
where the auxiliary kernels $R^{(j)}$, $1\leq j\leq n-1$, are given by
\begin{equation}\label{ioaTga-666}
R^{(j)}(x'):=x_j P^L(x'),\,\,\mbox{ for every }\,\,x'\in{\mathbb{R}}^{n-1}.
\end{equation}
Superficially, the terms in the right-most side of \eqref{ioaTga-555} appear to
have the same type of structure as the original function $u$ in \eqref{tghn-jan-2}
(since $\partial_j f\in L^p({\mathbb{R}}^{n-1})$), which raises the prospect of
handling them as in \eqref{tghn-jan-3}. However, such optimism is not justified
since the auxiliary kernels $R^{(j)}$ have a fundamentally different behavior at
infinity than the original $P^L$. Concretely, in place of \eqref{tghn-jan-1} we
now have
\begin{equation}\label{tghn-jan-1B}
\big|R^{(j)}_t(x')\big|\leq C\frac{|x_j|}{(t^2+|x'|^2)^{\frac{n}{2}}},
\qquad\forall\,x'\in{\mathbb{R}}^{n-1},\quad\forall\,t>0.
\end{equation}
In particular, $R^{(j)}_t(x')$ only decays as $|x'|^{1-n}$ at infinity, for each
$t>0$ fixed, so the analogue of \eqref{tghn-jan-3} in this case, i.e., the pointwise estimate
\begin{equation}\label{tghaa-1D}
{\mathcal{N}}\big(\partial_t u\big)\leq C{\mathcal{M}}(\nabla'f)
\quad\mbox{in}\quad{\mathbb{R}}^{n-1},
\end{equation}
where $\nabla'$ denotes the gradient in ${\mathbb{R}}^{n-1}$, is rendered hopeless.
This being said, the usual technology used in the proof of Cotlar's inequality may
be employed to show that in place of \eqref{tghaa-1D} one nonetheless has
\begin{equation}\label{tghn-PPP}
{\mathcal{N}}\big(\partial_t u\big)\leq C\sum_{j=1}^{n-1}
T^{(j)}_\star (\partial_j f)+C{\mathcal{M}}(\nabla'f)
\quad\mbox{in}\quad{\mathbb{R}}^{n-1},
\end{equation}
where $T^{(j)}_\star$ is the maximal singular integral operator acting on a generic
function $g$ defined in ${\mathbb{R}}^{n-1}$ according to
\begin{equation}\label{tghn-PPP.2}
T^{(j)}_\star g(x'):=\sup_{\varepsilon>0}\left|\int_{|x'-y'|>\varepsilon}k_j(x'-y')g(y')\,dy'
\right|,\qquad x'\in{\mathbb{R}}^{n-1},
\end{equation}
where the kernel $k_j$ is given by
\begin{equation}\label{tghn-PPP.3}
k_j(x'):=x_j\partial_t\big[P^L_t(x')\big]\Big|_{t=0},\qquad x'\in{\mathbb{R}}^{n-1}\setminus\{0'\}.
\end{equation}
In concert with the fact that each $k_j$ has the right amount of regularity and homogeneity, i.e.,
\begin{equation}\label{tghn-PPP.4}
\begin{array}{c}
k_j\in{\mathscr{C}}^\infty({\mathbb{R}}^{n-1}\setminus\{0'\}),\quad
k_j(\lambda x')=\lambda^{1-n}k_j(x')
\\[5pt]
\mbox{for every }\,\,\lambda>0\,\,\mbox{ and every }\,\,
x'\in{\mathbb{R}}^{n-1}\setminus\{0'\},
\end{array}
\end{equation}
estimate \eqref{tghn-PPP} then steers the proof of bounding the $L^p$ norm of
${\mathcal{N}}\big(\partial_t u\big)$ in the direction of Calder\'on-Zygmund theory.
However, what is needed for the latter to apply is a suitable cancellation condition
for the kernels $k_j$, say
\begin{equation}\label{tghn-PPP.5}
\int_{S^{n-2}}k_j(\omega')\,d\omega'=0,\qquad\forall\,j\in\{1,...,n-1\}.
\end{equation}
Under the mere ellipticity assumption on $L$ there is no reason to expect that
a cancellation condition such as \eqref{tghn-PPP.5} happens, so extra assumptions,
of an algebraic nature, need to be imposed to ensure its validity. In the sequel,
we identify a class of operators (cf. Definition~\ref{UaalpIKL}) for which the
respective kernels $k_j$ are odd, thus \eqref{tghn-PPP.5} holds. A natural issue
to consider is whether condition \eqref{tghn-PPP.5} would, on its own, ensure
well-posedness for \eqref{Dir-BVP-p:l}. The answer is no, as it may be seen by
looking at the case of \eqref{Dir-BVP-p:l} with $\ell=2$. This time, the boundary
datum $f$ is assumed to belong to $L^p_2({\mathbb{R}}^{n-1})$ and one is required
to estimate the $L^p$ norm of ${\mathcal{N}}(\partial^2_t u)$. By running the above
procedure, one now obtains (based on \eqref{ioaTga-444} and \eqref{ioaTga-555})
\begin{align}\label{ioaTga-5aa}
\partial^2_t\big[u(x',t)\big] &=
-\sum_{j=1}^{n-1}\partial_t\Big[R^{(j)}_t\ast\big(\partial_jf\big)\Big](x')
\nonumber\\[0pt]
&=\sum_{i=1}^{n-1}\sum_{j=1}^{n-1}\Big[R^{(ij)}_t\ast\big(\partial_i\partial_jf\big)\Big](x')
\,\,\mbox{ for every }\,\,(x',t)\in{\mathbb{R}}^n_{+},
\end{align}
where the second generation auxiliary kernels $R^{(ij)}$, $1\leq i,j\leq n-1$,
are given by
\begin{equation}\label{ioaTga-6bb}
R^{(ij)}(x'):=x_ix_j P^L(x'),\,\,\mbox{ for every }\,\,x'\in{\mathbb{R}}^{n-1}.
\end{equation}
However, these kernels exhibit a worse decay condition at infinity than
their predecessors in \eqref{tghn-jan-1B}, since now we only have
\begin{equation}\label{tghn-jan-1C}
\big|R^{(ij)}_t(x')\big|\leq C\frac{|x_ix_j|}{(t^2+|x'|^2)^{\frac{n}{2}}},
\qquad\forall\,x'\in{\mathbb{R}}^{n-1},\quad\forall\,t>0.
\end{equation}
This rules out, from the outset, the possibility of involving the Calder\'on-Zygmund
theory in the proceedings, thus rendering condition \eqref{tghn-PPP.5} irrelevant
for the case $\ell=2$. Of course, in the context of larger values of $\ell$ one
is faced with similar issues.

In summary, an approach based solely on generic qualitative properties of elliptic
second order operators runs into insurmountable difficulties, and the above
analysis makes the case for the necessity of additional algebraic assumptions
on the nature of the operator $L$, without which the well-posedness of
\eqref{Dir-BVP-p:l} is not generally expected for all $\ell\in{\mathbb{N}}_0$.

In this paper, we identify a large class of second order elliptic operators for
which a version of the procedure outlined above may be successfully implemented.
Using a piece of terminology formulated precisely in the body of the paper,
these are the operators $L$ possessing a distinguished coefficient tensor
(see Definition~\ref{UaalpIKL}). Under such a condition, the auxiliary kernels
referred to earlier become manageable and this eventually leads to the well-posedness
of the higher order regularity Dirichlet problem as formulated in \eqref{Dir-BVP-p:l}.
See Theorem~\ref{them:Dir-l} which is the main result of the paper.
In the last section, we illustrate the scope of the techniques developed here
by proving that such an approach works for any constant (complex) coefficient
scalar elliptic operator, as well as for the Lam\'e system of elasticity.
In fact, even in the case of the Laplacian, our well-posedness result for the
higher order Dirichlet problem in the upper-half space is new. In closing, we
also point out that the same circle of ideas works equally well for other
partial differential equations of basic importance in mathematical physics,
such as the Stokes system of hydrodynamics, the Maxwell system of electromagnetics,
and the Dirac operator of quantum theory (more on this may be found in the
forthcoming monograph \cite{MaMiMiMi}).

\section{Preliminaries}
\setcounter{equation}{0}
\label{S-2}

Throughout, we let ${\mathbb{N}}$ stand for the collection of all strictly
positive integers, and set ${\mathbb{N}}_0:={\mathbb{N}}\cup\{0\}$.
Also, fix $n\in{\mathbb{N}}$ with $n\geq 2$. We shall work in the upper-half space
\begin{equation}\label{RRR-UpHs}
{\mathbb{R}}^{n}_{+}:=\big\{x=(x',x_n)\in
{\mathbb{R}}^{n}={\mathbb{R}}^{n-1}\times{\mathbb{R}}:\,x_n>0\big\},
\end{equation}
whose topological boundary $\partial{\mathbb{R}}^{n}_{+}={\mathbb{R}}^{n-1}\times\{0\}$
will be frequently identified with the horizontal hyperplane ${\mathbb{R}}^{n-1}$
via $(x',0)\equiv x'$. Fix a number $\kappa>0$ and for each boundary point
$x'\in\partial{\mathbb{R}}^{n}_{+}$ introduce the conical nontangential approach region
\begin{equation}\label{NT-1}
\Gamma(x'):=\Gamma_\kappa(x'):=\big\{y=(y',t)\in{\mathbb{R}}^{n}_{+}:\,
|x'-y'|<\kappa\,t\big\}.
\end{equation}
Given a vector-valued function $u:{\mathbb{R}}^{n}_{+}\to{\mathbb{C}}^M$,
define the nontangential maximal function of $u$ by
\begin{equation}\label{NT-Fct}
\big({\mathcal{N}}u\big)(x'):=\big({\mathcal{N}}_\kappa u\big)(x')
:=\sup\,\big\{|u(y)|:\,y\in\Gamma_\kappa(x')\},\qquad x'\in\partial{\mathbb{R}}^{n}_{+}.
\end{equation}
As is well-known, for every $\kappa,\kappa'>0$ and $p\in(0,\infty)$ there exist
finite constants $C_0,C_1>0$ such that
\begin{equation}\label{N-Nal}
C_0\|{\mathcal{N}}_\kappa u\|_{L^p(\partial{\mathbb{R}}^{n}_{+})}
\leq\|{\mathcal{N}}_{\kappa'}\,u\|_{L^p(\partial{\mathbb{R}}^{n}_{+})}
\leq C_1 \|{\mathcal{N}}_\kappa u\|_{L^p(\partial{\mathbb{R}}^{n}_{+})},
\end{equation}
for each function $u$. Whenever meaningful, we also define
\begin{equation}\label{nkc-EE-2}
u\Big|^{{}^{n.t.}}_{\partial{\mathbb{R}}^{n}_{+}}(x')
:=\lim_{\Gamma_{\kappa}(x')\ni y\to (x',0)}u(y)
\quad\mbox{for }\,x'\in\partial{\mathbb{R}}^{n}_{+}.
\end{equation}

For each $p\in(1,\infty)$ and $k\in\mathbb{N}_0$ denote by $L^p_k(\mathbb{R}^{n-1})$
the classical Sobolev space of order $k$ in $\mathbb{R}^{n-1}$, consisting of
functions from $L^p(\mathbb{R}^{n-1})$ whose distributional derivatives up to order
$k$ are in $L^p(\mathbb{R}^{n-1})$. This becomes a Banach space when equipped with
the natural norm
\begin{equation}\label{Lnanb}
\|f\|_{L^p_k(\mathbb{R}^{n-1})}:=\|f\|_{L^p(\mathbb{R}^{n-1})}
+\sum_{|\alpha|\le k}\|\partial^\alpha f\|_{L^p(\mathbb{R}^{n-1})},\qquad
\forall\,f\in L^p_k(\mathbb{R}^{n-1}).
\end{equation}

Let $L$ be an elliptic operator as in \eqref{L-def}-\eqref{L-ell.X}.
Call $A:=\bigl(a^{\alpha\beta}_{rs}\bigr)_{\alpha,\beta,r,s}$ the
{\tt coefficient tensor} of $L$. To emphasize the dependence of $L$ on $A$,
let us agree to write $L_A$ in place of $L$ whenever necessary. In general,
there are multiple ways of expressing a given system $L$ as in \eqref{L-def}.
Indeed, if for any given $A=\bigl(a^{\alpha\beta}_{rs}\bigr)_{\alpha,\beta,r,s}$,
we define $A_{\rm sym}:=\Bigl(\tfrac{1}{2}
\bigl(a^{\alpha\beta}_{rs}+a^{\alpha\beta}_{sr}\bigr) \Bigr)_{\alpha,\beta,r,s}$, then
\begin{equation}\label{UF-HB.85.b}
L_{A_1}=L_{A_2}\,\Longleftrightarrow\,(A_1-A_2)_{\rm sym}=0.
\end{equation}
These considerations suggest introducing
\begin{equation}\label{AL-DDD}
{\mathfrak{A}}_L:=\Big\{A=\bigl(a_{rs}^{\alpha\beta}\bigr)
_{\substack{1\leq r,s\leq n\\[1pt] 1\leq\alpha,\beta\leq M}}\in{\mathbb{C}}^{\,nM}
\times{\mathbb{C}}^{\,nM}:\,L=L_A\Big\}.
\end{equation}
It follows from \eqref{UF-HB.85.b} that if the original coefficient tensor of $L$
satisfies the Legendre-Hadamard ellipticity condition \eqref{L-ell.X} then
any other coefficient tensor in ${\mathfrak{A}}_L$ does so. In other words,
the Legendre-Hadamard ellipticity condition is an intrinsic property of
the differential operator being considered, which does not depend on the
choice of a coefficient tensor used to represent this operator.

Given a system $L$ as in \eqref{L-def}, let $L^\top$ be the {\tt transposed}
of $L$, i.e., the $M\times M$ system of differential operators satisfying
\begin{equation}\label{Lam-TFF73}
\int_{{\mathbb{R}}^n}\langle Lu,v\rangle\,d{\mathscr{L}}^n
=\int_{{\mathbb{R}}^n}\big\langle u,L^\top v\big\rangle\,d{\mathscr{L}}^n,\qquad
\forall\,u,v\in{\mathscr{C}}_c^\infty\bigl({\mathbb{R}}^n\bigr)
,\,\mbox{${\mathbb{C}}^M$-valued},
\end{equation}
where ${\mathscr{L}}^n$ stands for the Lebesgue measure in ${\mathbb{R}}^n$.
A moment's reflection then shows that, if $L$ is as in \eqref{L-def}, then
\begin{equation}\label{L-def-Tra}
L^\top u=\Bigl(\partial_r(a^{\beta\alpha}_{sr}\partial_s u_\beta)
\Bigr)_{1\leq\alpha\leq M},\qquad\forall\,
u=(u_\beta)_{1\leq\beta\leq M}\in{\mathscr{C}}^2({\mathbb{R}}^n).
\end{equation}
That is, if $A^{\top}:=\bigl(a^{\beta\alpha}_{sr}\bigr)_{\substack{1\leq r,s\leq n
\\[1pt] 1\leq\alpha,\beta\leq M}}$ denotes the transpose of
$A=\bigl(a^{\alpha\beta}_{rs}\bigr)_{\substack{1\leq r,s\leq n \\[1pt] 1\leq\alpha,\beta\leq M}}$, formula \eqref{L-def-Tra} amounts to saying that
$\bigl(L_A\bigr)^{\top}=L_{A^\top}$.

The theorem below summarizes properties of a distinguished fundamental solution
of the operator $L$. It builds on the work carried out in various degrees of
generality in \cite[pp.\,72-76]{John}, \cite[p.\,169]{Hor-1}, \cite{Shapiro}, \cite[p.\,104]{Morrey}, and a proof in the present formulation may be found
in \cite{DM}, \cite{IMM}.

\begin{theorem}\label{FS-prop}
Assume that $L$ is an $M\times M$ elliptic, second order system in ${\mathbb{R}}^n$,
with complex constant coefficients as in \eqref{L-def}. Then there
exists a matrix $E=\bigl(E_{\alpha\beta}\bigr)_{1\leq\alpha,\beta\leq M}$
whose entries are tempered distribution in ${\mathbb{R}}^n$ and such that
the following properties hold:

\begin{list}{$(\theenumi)$}{\usecounter{enumi}\leftmargin=.8cm
\labelwidth=.8cm\itemsep=0.2cm\topsep=.1cm
\renewcommand{\theenumi}{\alph{enumi}}}
\item For each $\alpha,\beta\in\{1,...,M\}$, $E_{\alpha\beta}\in {\mathscr{C}}^\infty(\mathbb{R}^n \setminus\{0\})$ and $E_{\alpha\beta}(-x)=E_{\alpha\beta}(x)$ for all $x\in{\mathbb{R}}^n\setminus\{0\}$.
\item If $\delta_y$ stands for Dirac's delta distribution
with mass at $y$ then for each indices $\alpha,\beta\in\{1,...,M\}$, and every $x,y\in{\mathbb{R}}^n$,
\begin{equation}\label{fs}
\partial_{x_r}a^{\alpha\gamma}_{rs}\partial_{x_s}
\bigl[E_{\gamma\beta}(x-y)\bigr]=
\left\{
\begin{array}{ll}
0 & \mbox{ if }\,\,\alpha\neq\beta,
\\[4pt]
\delta_y(x) & \mbox{ if }\,\,\alpha=\beta.
\end{array}\right.
\end{equation}
\vskip 0.03in
\item For each $\alpha,\beta\in\{1,...,M\}$, one has
\begin{equation}\label{fs-str}
E_{\alpha\beta}(x)=\Phi_{\alpha\beta}(x)+c_{\alpha\beta}\ln|x|,
\qquad\forall\,x\in\mathbb{R}^n\setminus\{0\},
\end{equation}
where $\Phi_{\alpha\beta}\in {\mathscr{C}}^\infty(\mathbb{R}^n\setminus\{0\})$
is a homogeneous function of degree $2-n$, and the matrix
$\bigl(c_{\alpha\beta}\bigr)_{1\leq\alpha,\beta\leq M}\in{\mathbb{C}}^{M\times M}$
is identically zero when $n\geq 3$.
\item For each $\gamma\in\mathbb{N}_0^n$ there exists a finite constant
$C_\gamma>0$ such that for each $x\in{\mathbb{R}}^n\setminus\{0\}$
\begin{equation}\label{fs-est}
|\partial^\gamma E(x)|\leq
\left\{
\begin{array}{l}
\displaystyle\frac{C_\gamma}{|x|^{n+|\gamma|-2}}
\quad\mbox{ if either $n\geq 3$, or $n=2$ and $|\gamma|>0$},
\\[16pt]
C_0\bigl(1+\bigl|\ln|x|\bigr|\bigr)\quad\mbox{ if }\,\,n=2
\,\,\mbox{ and }\,\,|\gamma|=0.
\end{array}
\right.
\end{equation}
\item When restricted to ${\mathbb{R}}^n\setminus\{0\}$,
the (matrix-valued) distribution $\widehat{E}$ is a ${\mathscr{C}}^\infty$
function and, with ``hat" denoting the Fourier transform in ${\mathbb{R}}^n$,
\begin{equation}\label{E-ftXC}
\widehat{E}(\xi)=-\Bigl[\Bigl(\xi_r\xi_s a^{\alpha\beta}_{rs}\Bigr)_{
1\leq \alpha,\beta\leq M}\Bigr]^{-1}
\quad\mbox{for each}\quad\xi\in{\mathbb{R}}^n\setminus\{0\}.
\end{equation}
\item One can assign to each elliptic differential operator $L$ as in \eqref{L-def}
a fundamental solution $E^L$ which satisfies $(a)$--$(e)$ above and, in addition,
$\bigl(E^L\bigr)^\top=E^{L^\top}$, where the superscript $\top$ denotes transposition.
\item In the particular case $M=1$, i.e., in the situation when
$L={\rm div}A\nabla$ for some matrix
$A=(a_{rs})_{1\leq r,s\leq n}\in{\mathbb{C}}^{n\times n}$, an explicit formula
for the fundamental solution $E$ of $L$ is
\begin{equation}\label{YTcxb-ytSH}
E(x)=\left\{
\begin{array}{ll}
-\frac{1}{(n-2)\omega_{n-1}\sqrt{{\rm det}\,(A_{\rm sym})}}
\big\langle(A_{\rm sym})^{-1}x,x\big\rangle^{\frac{2-n}{2}}
& \mbox{ if }\,\,n\geq 3,
\\[12pt]
\frac{1}{4\pi\sqrt{{\rm det}\,(A_{\rm sym})}}
\log\bigl(\langle(A_{\rm sym})^{-1}x,x\rangle\bigr) & \mbox{ if }\,\,n=2,
\end{array}
\right.
\end{equation}
for $x\in{\mathbb{R}}^n\setminus\{0\}$. Here, $\log$ denotes the principal
branch of the complex logarithm function
(defined by the requirement that $z^t=e^{t\log z}$ holds for every
$z\in{\mathbb{C}}\setminus(-\infty,0]$ and every $t\in{\mathbb{R}}$).
\end{list}
\end{theorem}

\section{Poisson kernels}

In this section we discuss the notion of Poisson kernel in ${\mathbb{R}}^n_{+}$
for an operator $L$ as in \eqref{L-def}-\eqref{L-ell.X}. We also identify a
subclass of these Poisson kernels, which we call special Poisson kernels, that
plays a significant role in the treatment of boundary value problems.

\begin{definition}[Poisson kernel for $L$ in $\mathbb{R}^{n}_+$]\label{defi:Poisson}
Let $L$ be a second order elliptic system with complex coefficients as in
\eqref{L-def}-\eqref{L-ell.X}.
A {\tt Poisson kernel} for $L$ in $\mathbb{R}^{n}_+$ is a matrix-valued function $P=\big(P_{\alpha\beta}\big)_{1\leq\alpha,\beta\leq M}:
\mathbb{R}^{n-1}\to\mathbb{C}^{M\times M}$ such that:
\begin{list}{$(\theenumi)$}{\usecounter{enumi}\leftmargin=.8cm
\labelwidth=.8cm\itemsep=0.2cm\topsep=.1cm
\renewcommand{\theenumi}{\alph{enumi}}}
\item there exists $C\in(0,\infty)$ such that
$\displaystyle|P(x')|\leq\frac{C}{(1+|x'|^2)^{\frac{n}2}}$ for each
$x'\in\mathbb{R}^{n-1}$;
\item one has $\displaystyle\int_{\mathbb{R}^{n-1}}P(x')\,dx'=I_{M\times M}$,
the $M\times M$ identity matrix;
\item if $K(x',t):=P_t(x'):=t^{1-n}P(x'/t)$, for each
$x\in\mathbb{R}^{n-1}$ and $t>0$, then the function
$K=\big(K_{\alpha\beta}\big)_{1\leq\alpha,\beta\leq M}$
satisfies (in the sense of distributions)
\begin{equation}\label{uahgab-UBVCX}
LK_{\cdot\beta}=0\,\,\mbox{ in }\,\,\mathbb{R}^{n}_+
\,\,\mbox{ for each }\,\,\beta\in\{1,...,M\}.
\end{equation}
\end{list}
\end{definition}

\vskip 0.06in
\begin{remark}\label{Ryf-uyf}
The following comments pertain to Definition~\ref{defi:Poisson}.
\begin{list}{$(\theenumi)$}{\usecounter{enumi}\leftmargin=.8cm
\labelwidth=.8cm\itemsep=0.2cm\topsep=.1cm
\renewcommand{\theenumi}{\roman{enumi}}}
\item Condition $(a)$ ensures that the integral in part $(b)$ is
absolutely convergent.
\item From $(a)$ and $(b)$ one can easily check that for each $p\in(1,\infty]$
there exists a finite constant $C=C(c,M,n,p)>0$ with the property that if
$f\in L^p(\mathbb{R}^{n-1})$ and $u(x',t):=(P_t* f)(x')$ for
$(x',t)\in\mathbb{R}^{n}_+$, then
\begin{equation}\label{smetg}
\big\|{\mathcal{N}}u\big\|_{L^p(\partial\mathbb{R}^{n}_+)}\leq C
\|f\|_{L^p(\mathbb{R}^{n-1})}\,\,\mbox{ and }\,\,
u\Big|^{{}n.t.}_{\partial{\mathbb{R}}^n_{+}}=f\,\,\mbox{ a.e. in }\,\,{\mathbb{R}}^{n-1}.
\end{equation}
\item Condition $(c)$ and the ellipticity of the operator $L$ ensure that
$K\in{\mathscr{C}}^\infty(\mathbb{R}^{n}_+)$. Given that
$P(x')=K(x',1)$ for each point $x'\in{\mathbb{R}}^{n-1}$, we then deduce that
$P\in{\mathscr{C}}^\infty(\mathbb{R}^{n-1})$. Furthermore, via a direct calculation
it may be checked that
\begin{equation}\label{ioaTga}
\partial_t\big[P_t(x')\big]=-\sum_{j=1}^{n-1}\partial_{x_j}\Big[\frac{x_j}{t}P_t(x')\Big]
\,\,\mbox{ for every }\,\,(x',t)\in{\mathbb{R}}^n_{+}.
\end{equation}
\item Condition $(b)$ is equivalent to $\lim\limits_{t\to 0^{+}}P_t(x')
=\delta_{0'}(x')\,I_{M\times M}$ in ${\mathscr{D}}'({\mathbb{R}}^{n-1})$,
where $\delta_{0'}$ is Dirac's distribution with mass at the origin $0'$
of ${\mathbb{R}}^{n-1}$.
\end{list}
\end{remark}

Poisson kernels for elliptic boundary value problems in a half-space have
been studied extensively in \cite{ADNI}, \cite{ADNII}, \cite[\S{10.3}]{KMR2},
\cite{Sol}, \cite{Sol1}, \cite{Sol2}. Here we record a corollary of more general
work done by S.\,Agmon, A.\,Douglis, and L.\,Nirenberg in \cite{ADNII}.

\begin{theorem}\label{ya-T4-fav}
Any elliptic differential operator $L$ as in \eqref{L-def} has a Poisson kernel
$P$ in the sense of Definition~\ref{defi:Poisson}, which has the additional
property that the function $K(x',t):=P_t(x')$ for all $(x',t)\in{\mathbb{R}}^n_{+}$,
satisfies $K\in{\mathscr{C}}^\infty\big(\overline{{\mathbb{R}}^n_{+}}
\setminus B(0,\varepsilon)\big)$ for every $ \varepsilon>0 $ and
$K(\lambda x)=\lambda^{1-n}K(x)$ for all $x\in{\mathbb{R}}^n_{+}$ and $\lambda>0$.

Hence, in particular, for each $\alpha\in{\mathbb{N}}_0^n$ there exists
$C_\alpha\in(0,\infty)$ with the property that
$\big|(\partial^\alpha K)(x)\big|\leq C_\alpha\,|x|^{1-n-|\alpha|}$, for every $x\in{\overline{{\mathbb{R}}^n_{+}}}\setminus\{0\}$.
\end{theorem}

One important consequence of the existence of a Poisson kernel $P$ for an operator
$L$ in the upper-half space is that for every $f\in L^p(\mathbb{R}^{n-1})$
the convolution $(P_t* f)(x')$ for $(x',t)\in\mathbb{R}^{n}_+$, yields a solution
for the $L^p$-Dirichlet problem for $L$ in the upper-half space. Hence, the difficulty
in proving well-posedness for such a problem comes down to proving uniqueness.
In the case of the Laplacian, this is done by employing the maximum principle for
harmonic functions, a tool not available in the case of systems. In \cite{MaMiMiMi}
we overcome this difficulty by constructing an appropriate Green function associated
with the $L^p$-Dirichlet problem for $L$ in the upper-half space.

\begin{theorem}\cite{MaMiMiMi}\label{YF-TY.6yh}
For each $p\in(1,\infty)$ the $L^p$-Dirichlet boundary value problem
for $L$ in $\mathbb{R}^{n}_+$, that is, \eqref{Dir-BVP-p:l} with $\ell=0$,
has a unique solution $u=(u_\beta)_{1\leq\beta\leq M}$ satisfying, for some
finite $C=C(L,n,p)>0$,
\begin{equation}\label{Dir-BVP-p2F}
\big\|{\mathcal{N}}u\big\|_{L^p(\partial\mathbb{R}^{n}_+)}\leq C
\|f\|_{L^p(\mathbb{R}^{n-1})}.
\end{equation}
Moreover, the solution $u$ is given by
\begin{equation}\label{eqn:exp-u-Poisson}
u(x',t)=(P_t*f)(x')
=\Big(\int_{\mathbb{R}^{n-1}} \big(P_{\beta\alpha} \big)_t(x'-y')\,f_\alpha(y')\,dy'\Big)_{\beta}
\end{equation}
for all $(x',t)\in{\mathbb{R}}^n_{+}$, where $P$ is the Poisson
kernel from Theorem \ref{ya-T4-fav}.
\end{theorem}

A corollary of this theorem is the uniqueness of the Poisson kernel
for $L$ in ${\mathbb{R}}^n_{+}$.

\begin{proposition}\label{uniqueness-Poisson}
Any operator $L$ as in \eqref{L-def}-\eqref{L-ell.X} has a unique
Poisson kernel as in Definition~\ref{defi:Poisson} (which is the Poisson kernel
given by Theorem~\ref{ya-T4-fav}).
\end{proposition}

\begin{proof}
Suppose $L$ has two Poisson kernels, say $P$ and $Q$, in ${\mathbb{R}}^n_{+}$.
Then for each $p\in(1,\infty)$ and every $f\in L^p(\mathbb{R}^{n-1})$,
the function $u(x',t):=(P_t-Q_t)*f(x')$ for $(x',t)\in\mathbb{R}^n_{+}$, is
a solution of the homogeneous $L^p$-Dirichlet boundary value problem in
${\mathbb{R}}^n_{+}$. Hence, by Theorem~\ref{YF-TY.6yh}, $u=0$ in ${\mathbb{R}}^n_{+}$.
This forces $P=Q$ in $\mathbb{R}^{n-1}$.
\end{proof}

As mentioned before, there are multiple coefficient tensors which yield a given
system $L$ as in \eqref{L-def}. The following proposition paves the way for
singling out, in Definition~\ref{UaalpIKL} formulated a little later,
a special subclass among all these coefficient tensors.

\begin{proposition}\cite{MaMiMi}\label{HGfs59+E}
Assume that $A=\bigl(a_{rs}^{\alpha\beta}\bigr)
_{\substack{1\leq r,s\leq n\\ 1\leq\alpha,\beta\leq M}}$ is a coefficient
tensor with complex entries satisfying the Legendre-Hadamard ellipticity condition
\eqref{L-ell.X}. Let $L$ be the system associated with the given coefficient tensor
$A$ as in \eqref{L-def} and denote by $E=(E_{\gamma\beta})_{1\leq\gamma,\beta\leq M}$
the fundamental solution from Theorem~\ref{FS-prop} for the system $L$. Also, let
${\rm Symb}_L(\xi):=
-\Bigl(\xi_r\xi_s a^{\alpha\beta}_{rs}\Bigr)_{1\leq\alpha,\beta\leq M}$,
for $\xi\in{\mathbb{R}}^n\setminus\{0\}$, denote the symbol of the differential
operator $L$ and set
\begin{equation}\label{yfg-yL4f}
\bigl(S_{\gamma\beta}(\xi)\bigr)_{1\leq\gamma,\beta\leq M}
:=\Bigl[{\rm Symb}_L(\xi)\Bigr]^{-1}\in{\mathbb{C}}^{\,M\times M},
\qquad\forall\,\xi\in{\mathbb{R}}^n\setminus\{0\}.
\end{equation}
Then the following two conditions are equivalent.
\begin{list}{$(\theenumi)$}{\usecounter{enumi}\leftmargin=.8cm
\labelwidth=.8cm\itemsep=0.2cm\topsep=.1cm
\renewcommand{\theenumi}{\alph{enumi}}}

\item For each $s,s'\in\{1,...,n\}$ and each
$\alpha,\gamma\in\{1,...,M\}$ there holds
\begin{equation}\label{Ea4-fCii-n3}
\Bigl[a^{\beta\alpha}_{s's}-a^{\beta\alpha}_{ss'}
+\xi_r a^{\beta\alpha}_{rs}\partial_{\xi_{s'}}
-\xi_r a^{\beta\alpha}_{rs'}\partial_{\xi_{s}}\Bigr]
S_{\gamma\beta}(\xi)=0,\qquad\forall\,\xi\in{\mathbb{R}}^n\setminus\{0\},
\end{equation}
and (with $\sigma_{S^1}$ denoting the arc-length measure on $S^1$)
\begin{equation}\label{Ea4-fCii-n2B}
\int_{S^1}\Big(a^{\beta\alpha}_{rs}\xi_{s'}-a^{\beta\alpha}_{rs'}\xi_{s}\Big)
\big(\xi_r S_{\gamma\beta}(\xi)\big)\,d\sigma_{S^1}(\xi)=0\,\,\mbox{ if }\,\,n=2.
\end{equation}
\item There exists a matrix-valued function
$k=\bigl\{k_{\gamma\alpha}\bigr\}_{1\leq\gamma,\alpha\leq M}
:{\mathbb{R}}^{n}\setminus\{0\}\to{\mathbb{C}}^{M\times M}$
with the property that for each $\gamma,\alpha\in\{1,...,M\}$
and $s\in\{1,...,n\}$ one has
\begin{equation}\label{yfg-yLLL}
a^{\beta\alpha}_{rs}(\partial_r E_{\gamma\beta})(x)
=x_s k_{\gamma\alpha}(x)\,\,\mbox{ for all }\,\,
x\in{\mathbb{R}}^n\setminus\{0\}.
\end{equation}
\end{list}
\end{proposition}

\vskip 0.10in

In light of the properties of the fundamental solution, condition
\eqref{yfg-yLLL} readily implies that
\begin{equation}\label{kan-UHab-5}
k\in{\mathscr{C}}^\infty\bigl({\mathbb{R}}^{n}\setminus\{0\}\bigr)
\,\,\mbox{ and $k$ is even and homogeneous of degree $-n$}.
\end{equation}
Note that condition $(a)$ in Proposition~\ref{HGfs59+E} is entirely formulated
in terms of the coefficient tensor $A$. This suggests making the following
definition (recall that ${\mathfrak{A}}_L$ has been introduced in \eqref{AL-DDD}).

\begin{definition}\label{UaalpIKL}
Given a second-order elliptic system $L$ with constant complex coefficients
as in \eqref{L-def}-\eqref{L-ell.X}, call a coefficient tensor
\begin{equation}\label{kan-UHjnbbb}
A=\bigl(a_{rs}^{\alpha\beta}\bigr)
_{\substack{1\leq r,s\leq n\\ 1\leq\alpha,\beta\leq M}}\in{\mathfrak{A}}_L
\end{equation}
{\tt distinguished} provided condition $(a)$ in Proposition~\ref{HGfs59+E} holds,
and denote by ${\mathfrak{A}}_L^{dis}$ the totality of such distinguished
coefficient tensors for $L$, i.e.,
\begin{align}\label{AL-DDD.2}
{\mathfrak{A}}_L^{dis}:=\Big\{A=\bigl(a_{rs}^{\alpha\beta}\bigr)
_{\substack{1\leq r,s\leq n\\ 1\leq\alpha,\beta\leq M}}\in{\mathfrak{A}}_L:\,
& \mbox{conditions \eqref{Ea4-fCii-n3}-\eqref{Ea4-fCii-n2B} hold for each }\,\,
\nonumber\\[0pt]
& s,s'\in\{1,...,n\}\,\,\mbox{ and }\,\,\alpha,\gamma\in\{1,...,M\}\Big\}.
\end{align}
\end{definition}

\begin{remark}\label{Tavav-yag8}
We claim that ${\mathfrak{A}}_L^{dis}\not=\emptyset$ whenever $M=1$.
More specifically, when $M=1$, i.e., $L={\rm div}A\nabla$ with
$A=(a_{rs})_{1\leq r,s\leq n}\in{\mathbb{C}}^{n\times n}$, one has
$A_{\rm sym}\in {\mathfrak{A}}_L^{dis}$. To see that this is the case,
recall that checking the membership of $A_{\rm sym}$
to ${\mathfrak{A}}_L^{dis}$ comes down to verifying conditions
\eqref{Ea4-fCii-n3}-\eqref{Ea4-fCii-n2B} for the entries in the matrix $A_{\rm sym}$.
Note that for each index $s\in\{1,...,n\}$ we have in this case
\begin{equation}\label{AL-DDD.2ajan.2}
\partial_{\xi_s}\big[{\rm Symb}_L(\xi)\big]^{-1}=2\big[{\rm Symb}_L(\xi)\big]^{-2}
\big(A_{\rm sym}\xi\big)_s,\qquad\forall\,\xi\in{\mathbb{R}}^n\setminus\{0\},
\end{equation}
and \eqref{Ea4-fCii-n3} readily follows from this. Moreover, if $n=2$,
condition \eqref{Ea4-fCii-n2B} reduces to checking that
\begin{equation}\label{AL-DDD.2ajan.3}
\int_{S^1}\frac{\big(A_{\rm sym}\xi\big)\cdot(\xi_2,-\xi_1)}
{\big(A_{\rm sym}\xi\big)\cdot\xi}\,d\sigma_{S^1}(\xi)=0.
\end{equation}
The key observation in this regard is that if $f(\theta):=
\Big[\big(A_{\rm sym}\xi\big)\cdot\xi\Big]\Big|_{\xi=(\cos\theta,\,\sin\theta)}$ then
\begin{equation}\label{AL-DDD.2ajan.4}
\frac{\big(A_{\rm sym}\xi\big)\cdot(\xi_2,-\xi_1)}
{\big(A_{\rm sym}\xi\big)\cdot\xi}\Big|_{\xi=(\cos\theta,\,\sin\theta)}
=-\frac{f'(\theta)}{2f(\theta)},\qquad\forall\,\theta\in(0,2\pi).
\end{equation}
Now \eqref{AL-DDD.2ajan.3} readily follows from \eqref{AL-DDD.2ajan.4},
proving that indeed $A_{\rm sym}\in {\mathfrak{A}}_L^{dis}$.
\end{remark}

One of the main features of elliptic systems having a distinguished coefficient tensor
is that their Poisson kernels have a special form. This is made more precise in the
next proposition.

\begin{proposition}\cite{MaMiMiMi}\label{uniq:double->!Poisson}
Let $L$ be a constant coefficient system as in \eqref{L-def}-\eqref{L-ell.X}.
Assume that ${\mathfrak{A}}_L^{dis}\not=\emptyset$ and let
$k=\bigl\{k_{\gamma\alpha}\bigr\}_{1\leq\gamma,\alpha\leq M}:
{\mathbb{R}}^{n}\setminus\{0\}\to{\mathbb{C}}^{M\times M}$ be the function
appearing in condition $(b)$ of Proposition~\ref{HGfs59+E}. Then the unique Poisson
kernel for $L$ in $\mathbb{R}^{n}_+$ from Theorem \ref{ya-T4-fav} has the form
\begin{equation}\label{Yagav-g45}
P(x')=2k(x',1),\qquad\forall\,x'\in{\mathbb{R}}^{n-1}.
\end{equation}
\end{proposition}

\section{The Dirichlet problem with data in higher order Sobolev spaces}
\setcounter{equation}{0}
\label{S-6}

The main result of our paper is the following theorem giving the well-posedness
of the Dirichlet boundary value problem in ${\mathbb{R}}^n_{+}$ with data in
higher-order Sobolev spaces for constant (complex) coefficient elliptic systems
possessing a distinguished coefficient tensor.

\begin{theorem}\label{them:Dir-l}
Let $L$ be an operator as in \eqref{L-def}-\eqref{L-ell.X} with the property
that ${\mathfrak{A}}_L^{dis}\not=\emptyset$, and fix $p\in(1,\infty)$ and
$\ell\in\mathbb{N}_0$. Then the $\ell$-th order Dirichlet boundary value problem
for $L$ in $\mathbb{R}^{n}_+$,
\begin{equation}\label{Dir-BVP-pLL:l}
\left\{
\begin{array}{l}
Lu=0\,\,\mbox{ in }\,\,\mathbb{R}^{n}_+,
\\[6pt]
{\mathcal{N}}(\nabla^k u)\in L^p(\partial\mathbb{R}^{n}_+),\,\,0\le k\leq\ell,
\\[6pt]
u\bigl|_{\partial\mathbb{R}^{n}_{+}}^{{}^{n.t.}}=f\in L^p_\ell(\mathbb{R}^{n-1}),
\end{array}
\right.
\end{equation}
has a unique solution. Moreover, the solution $u$ of \eqref{Dir-BVP-pLL:l}
is given by
\begin{equation}\label{eqn-Dir-l:u}
u(x',t)=(P_t*f)(x'),\qquad\forall\,(x',t)\in{\mathbb{R}}^n_{+},
\end{equation}
where $P$ is the Poisson kernel for $L$ in $\mathbb{R}^{n}_+$ from
Theorem~\ref{ya-T4-fav}. Furthermore, there exists a
constant $C=C(n,p,L,\ell)\in(0,\infty)$ with the property that
\begin{equation}\label{Dir-BVP-p2F.ii:l}
\sum_{k=0}^\ell \big\|{\mathcal{N}}(\nabla^k u)\big\|_{L^p(\partial\mathbb{R}^{n}_+)}
\leq C\|f\|_{L^p_\ell(\mathbb{R}^{n-1})}.
\end{equation}
\end{theorem}

The remainder of this section is devoted to providing a proof for
Theorem~\ref{them:Dir-l}. This requires developing a number of tools,
which are introduced and studied first.

To fix notation let $\nabla_{x'}:=(\partial_1,\dots,\partial_{n-1})$ and,
alternatively, use $\partial_t$ in place of $\partial_n$ if the description
$(x',t)$ of points in ${\mathbb{R}}^{n-1}\times(0,\infty)$ is emphasized in
place of $x\in{\mathbb{R}}^n_{+}$. Also fix $p\in(1,\infty)$, $\ell\in{\mathbb{N}}$,
and let $f\in L^p_\ell(\mathbb{R}^{n-1})$. In view of Theorem~\ref{YF-TY.6yh},
proving Theorem~\ref{them:Dir-l} reduces to showing that the function
$u(x',t)=(P_t*f)(x')$ for $(x',t)\in{\mathbb{R}}^n_{+}$ satisfies
${\mathcal{N}}(\nabla^k u)\in L^p(\partial\mathbb{R}^{n}_+)$ for
$k=1,\dots,\ell$, as well as \eqref{Dir-BVP-p2F.ii:l}.
Suppose $\alpha=(\alpha_1,...,\alpha_n)\in{\mathbb{N}}_0$ is such that $|\alpha|\leq\ell$.
It is immediate that if $\alpha_n=0$ then
$\partial^\alpha u(x',t)=\big(P_t*(\partial^\alpha f)\big)(x')$
for $(x',t)\in{\mathbb{R}}^n_{+}$. The crux of the matter is handling
$\partial^\alpha u$ when $\alpha_n\not=0$. As you will see below, the special
format of the Poisson kernel guaranteed by Proposition~\ref{uniq:double->!Poisson}
allows us to prove a set of basic identities expressing
$\partial_t^k \big[(P_t* f)(x')\big]$ as a linear combination of
$(P_t* \nabla^k_{x'} f)(x')$ and convolutions of certain auxiliary kernels
with derivatives of $f$. Here is the class of auxiliary kernels just alluded to.

\begin{definition}\label{kernelsQj-D}
Given an operator $L$ as in \eqref{L-def}-\eqref{L-ell.X} denote by $E$ the
fundamental solution for $L$ from Theorem~\ref{FS-prop}. Then for each
$j\in\{1,\dots,n\}$ define the auxiliary matrix-valued kernel function
\begin{equation}\label{kernelQ-1}
Q^{(j)}(x')
:=\Big(Q^{(j)}_{\alpha\beta}(x')\Big)_{1\leq\alpha,\beta\leq M}
:=\Big((\partial_j E_{\alpha\beta})(x',1)\Big)_{1\leq\alpha,\beta\leq M},
\qquad\forall\,x'\in{\mathbb{R}}^{n-1}.
\end{equation}
\end{definition}

In the next lemma we describe some of the basic
properties of the auxiliary kernels just introduced.

\begin{lemma}\label{RTHb}
Let $L$ be an operator as in \eqref{L-def}-\eqref{L-ell.X} and let
$\big\{Q^{(j)}_{\alpha\beta}\big\}_{j,\alpha,\beta}$ be the family of functions
from \eqref{kernelQ-1}. Then the following are true.
\begin{list}{$(\theenumi)$}{\usecounter{enumi}\leftmargin=.8cm
\labelwidth=.8cm\itemsep=0.2cm\topsep=.1cm
\renewcommand{\theenumi}{\alph{enumi}}}
\item There exists some constant $C=C(n,L)\in(0,\infty)$ such that for each indices
$j\in\{1,\dots,n\}$ and $\alpha,\beta\in\{1,\dots,M\}$ one has
\begin{equation}\label{kernelQ-2}
Q^{(j)}_{\alpha\beta}\in {\mathscr{C}}^\infty(\mathbb{R}^{n-1})\,\,\mbox{ and }\,\,
\Big|Q^{(j)}_{\alpha\beta}(x')\Big|\leq\frac{C}{(|x'|+1)^{n-1}}\quad
\forall\,x'\in{\mathbb{R}}^{n-1}.
\end{equation}
\item For each $j,r\in\{1,\dots,n\}$ and every
$\alpha,\gamma\in\{1,\dots,M\}$ we have
\begin{equation}\label{kernelQ-4}
\partial_j\Big[\big(Q^{(r)}_{\alpha\gamma}\big)_t(x')\Big]
=\partial_r\Big[\big(Q^{(j)}_{\alpha\gamma}\big)_t(x')\Big],\qquad
\forall\,(x',t)\in{\mathbb{R}}^n_{+}.
\end{equation}
\item Given any $f\in L^p({\mathbb{R}}^{n-1})$ where $p\in(1,\infty)$,
along with $j\in\{1,\dots,n\}$ and $\alpha,\beta\in\{1,\dots,M\}$, define the function
\begin{equation}\label{GTdS}
u^{(j)}_{\alpha\beta}:{\mathbb{R}}^n_{+}\to {\mathbb{C}},\quad
u^{(j)}_{\alpha\beta}(x',t):=\Big[\big(Q^{(j)}_{\alpha\beta}\big)_t\ast f\Big](x'),
\quad\forall\,(x',t)\in{\mathbb{R}}^n_{+}.
\end{equation}
Then there exists a constant $C\in(0,\infty)$ independent of $f$ such that
\begin{equation}\label{GTdS-2}
\big\|{\mathcal{N}}u^{(j)}_{\alpha\beta}\big\|_{L^p({\mathbb{R}}^{n-1})}
\leq C\|f\|_{L^p({\mathbb{R}}^{n-1})}.
\end{equation}
\end{list}
\end{lemma}

\begin{proof}
Let $E$ be the fundamental solution for $L$ defined in Theorem~\ref{FS-prop}.
The fact that the claims in $(a)$ hold is a consequence of \eqref{kernelQ-1},
and Theorem \ref{FS-prop} parts $(a)$ and $(d)$. Next, fix $j\in\{1,\dots,n\}$,
$\alpha,\beta\in\{1,\dots,M\}$ and let $(x',t)\in{\mathbb{R}}^n_{+}$. Since
$\nabla E$ is positive homogeneous of order $1-n$ in ${\mathbb{R}}^n\setminus\{0\}$
(cf. property $(c)$ in Theorem~\ref{FS-prop}), one has
\begin{equation}\label{efgug-3}
\big(Q^{(r)}_{\gamma\beta}\big)_t(x')
=t^{1-n}(\partial_r E_{\gamma\beta})(x'/t,1)
=\big(\partial_r E_{\gamma\beta}\big)(x',t),\qquad\forall\,r\in\{1,\dots,n\}.
\end{equation}
Now \eqref{efgug-3} and the first condition in \eqref{kernelQ-2} imply that for every
$j,r\in\{1,\dots,n\}$,
\begin{equation}\label{efgug-4}
\partial_j\Big[\big(Q^{(r)}_{\alpha\gamma}\big)_t(x')\Big]
=\big(\partial_j\partial_r E_{\gamma\beta}\big)(x',t)
=\big(\partial_r\partial_j E_{\gamma\beta}\big)(x',t)
=\partial_r\Big[\big(Q^{(j)}_{\alpha\gamma}\big)_t(x')\Big],
\end{equation}
proving \eqref{kernelQ-4}.

There remains to prove the claim in $(c)$. To this end, let
$f\in L^p({\mathbb{R}}^{n-1})$ for some $p\in(1,\infty)$.
Then by \eqref{GTdS} and \eqref{efgug-3} we have
\begin{equation}\label{GTdS-3}
u^{(j)}_{\alpha\beta}(x',t)=\int_{{\mathbb{R}}^{n-1}}
(\partial_j E_{\alpha\beta})(x'-y',t)f(y')\,dy',
\qquad\forall\,(x',t)\in{\mathbb{R}}^n_{+}.
\end{equation}
If we now write $K=\partial_j E_{\alpha\beta}$, the properties of $E$
(cf. Theorem~\ref{FS-prop}) imply that
$K\in\mathscr{C}^\infty({\mathbb{R}}^{n}\setminus\{0\})$ with $K(-x)=-K(x)$ and
$K(\lambda\,x)=\lambda^{-(n-1)}K(x)$ for every $\lambda>0$ and $x\in{\mathbb{R}}^{n}\setminus\{0\}$. We can therefore invoke standard
Calder\'on-Zygmund theory and conclude that \eqref{GTdS-2} holds.
\end{proof}

In order to elaborate on the relationship between the family of auxiliary
kernels from Definition~\ref{kernelsQj-D} and the Poisson kernel
for the operator $L$ in ${\mathbb{R}}^n_{+}$, under the assumption
${\mathfrak{A}}_L^{dis}\not=\emptyset$, we first need to introduce some notation
which facilitates the subsequent discussion. Specifically,
given a coefficient tensor $A=\bigl(a_{rs}^{\alpha\beta}\bigr)_{r,s,\alpha,\beta}$ with complex entries satisfying the Legendre-Hadamard ellipticity condition \eqref{L-ell.X},
for each $r,s\in\{1,\dots,n\}$ abbreviate
\begin{equation}\label{eq2.13ttt.XX}
A_{rs}:=\Bigl(a^{\alpha\beta}_{rs}\Bigr)_{1\leq\alpha,\beta\leq M}.
\end{equation}
Note that the ellipticity condition \eqref{L-ell.X} written for
$\xi:={\bf e}_n\in{\mathbb{R}}^n$ yields, in particular, that  $A_{nn}=\Bigl(a^{\alpha\beta}_{nn}\Bigr)_{1\leq\alpha,\beta\leq M}
\in{\mathbb{C}}^{M\times M}$ is an invertible matrix. Next, for each sufficiently
smooth vector field $u=(u_\beta)_{1\leq\beta\leq M}$, define
\begin{equation}\label{Tang-partT-2}
D_{A}u:=\Bigl(a_{ns}^{\alpha\beta}\partial_s u_\beta\Bigr)_{1\leq\alpha\leq M},
\end{equation}
and set (with the superscript $\top$ denoting transposition)
\begin{equation}\label{Tang-partT-2B}
\partial_{\rm tan} u:=-\big(A^{\top}_{nn}\big)^{-1}\,
\Big[\Big(\sum_{s=1}^{n-1}a_{sn}^{\beta\alpha}\partial_s u_\beta\Big)
_{1\leq\alpha\leq M}\Big].
\end{equation}
The notation $\partial_{\rm tan}$ is justified by the fact that its expression only
involves partial derivatives in directions tangent to the boundary of the upper-half
space $\partial{\mathbb{R}}^n_{+}$.

For reasons that will become clear momentarily, we are interested in decomposing
the operator $\partial_t(=\partial_n)$ as the sum between a linear combination
of the partial derivative operators $\partial_j$, $j=1,\dots,n-1$, (which correspond
to tangential directions to $\partial{\mathbb{R}}^n_{+}$) and a suitable (matrix)
multiple of $D_{A^{\!\top}}$.

\begin{lemma}\label{decomp-pt}
One has $\partial_t=\partial_{\rm tan}+\big(A_{nn}^\top\big)^{-1}D_{A^{\!\top}}.$
\end{lemma}

\begin{proof}
Given $u=(u_\beta)_{1\leq\beta\leq M}\in {\mathscr{C}}^1({\mathbb{R}}^n_{+})$ we may write
\begin{align}\label{eqn:partial-t:tan}
\partial_t u-\big(A_{nn}^\top\big)^{-1}D_{A^{\!\top}}u
&=\big(A^{\top}_{nn}\big)^{-1}\,
\Big[A^{\top}_{nn}\partial_t u-D_{\,A^{\!\top}}\,u\Big]
\nonumber\\[4pt]
&=\big(A^{\top}_{nn}\big)^{-1}\,
\Big[\big(a_{nn}^{\beta \alpha} \partial_t u_\beta-a_{sn}^{\beta\alpha}
\partial_s u_\beta\big)_{1\leq\alpha\leq M}\Big]
\nonumber\\[4pt]
&=-\big(A^{\top}_{nn}\big)^{-1}\,
\Big[\Big(\sum_{s=1}^{n-1}a_{sn}^{\beta\alpha}\partial_s u_\beta\Big)
_{1\leq\alpha\leq M}\Big]
=\partial_{\rm tan}u,
\end{align}
as desired.
\end{proof}

We are now ready to state and prove a number of basic identities relating
the family of auxiliary kernels from Definition~\ref{kernelsQj-D} to the
Poisson kernel for the operator $L$, under the assumption that the latter
has a distinguished coefficient tensor.

\begin{proposition}\label{kernelsQj}
Let $L$ be an operator as in \eqref{L-def}-\eqref{L-ell.X} with the property that
${\mathfrak{A}}_L^{dis}\not=\emptyset$. Denote by $P$ the Poisson kernel for
$L$ from Theorem~\ref{ya-T4-fav} and fix some coefficient tensor
\begin{equation}\label{ytar-yR5}
A=\bigl(a_{rs}^{\alpha\beta}\bigr)
_{\substack{1\leq r,s\leq n\\ 1\leq\alpha,\beta\leq M}}\in{\mathfrak{A}}_L^{dis}.
\end{equation}
Then the auxiliary kernels $\big\{Q^{(j)}_{\alpha\beta}\big\}_{j,\alpha,\beta}$
introduced in Definition~\ref{kernelsQj-D} satisfy the following properties:
\begin{list}{$(\theenumi)$}{\usecounter{enumi}\leftmargin=.8cm
\labelwidth=.8cm\itemsep=0.2cm\topsep=.1cm
\renewcommand{\theenumi}{\alph{enumi}}}
\item for each $\alpha,\gamma\in\{1,\dots,M\}$ one has
for every $x'\in{\mathbb{R}}^{n-1}$ and every $t=x_n>0$
\begin{equation}\label{kernelQ-3}
2a^{\beta\alpha}_{rs}\Big(Q^{(r)}_{\gamma\beta}\Big)_t(x')=
\frac{x_s}{t}\big(P_{\gamma\alpha}\big)_t(x')\,\,\mbox{ for each }\,\,s\in\{1,\dots,n\};
\end{equation}
\item for every $\alpha,\gamma\in\{1,\dots,M\}$ one has
for every $x'\in{\mathbb{R}}^{n-1}$ and every $t>0$
\begin{equation}\label{kernelQ-3BB}
\partial_t\Big[\big(P_{\gamma\alpha}\big)_t(x')\Big]
=-2\sum\limits_{s=1}^{n-1}a^{\beta\alpha}_{rs}\partial_s
\Big[\big(Q^{(r)}_{\gamma\beta}\big)_t(x')\Big];
\end{equation}
\item for each $\gamma\in\{1,\dots,M\}$ one has
\begin{align}\label{kernelQ-4BB}
\big(Q^{(n)}_{\gamma\alpha}\big)_{1\leq\alpha\leq M}
=\,&\tfrac{1}{2}\big(A_{nn}^{\top}\big)^{-1}
\Big(\big(P_{\gamma\mu}\big)_{1\leq\mu\leq M}\Big)
\nonumber\\[4pt]
&-\sum_{s=1}^{n-1}\big(A_{nn}^{\top}\big)^{-1}\,
\Big(\big(a_{sn}^{\beta\mu}Q^{(s)}_{\gamma\beta}\big)_{1\leq\mu\leq M}\Big)
\,\,\mbox{ in }\,\,{\mathbb{R}}^{n-1}.
\end{align}
\end{list}
\end{proposition}

\begin{proof}
Since ${\mathfrak{A}}_L^{dis}\not=\emptyset$, Proposition~\ref{HGfs59+E}
ensures that the Poisson kernel $P$ satisfies \eqref{Yagav-g45}.
Hence, if $E$ is the fundamental solution for $L$ from Theorem~\ref{FS-prop},
starting with \eqref{kernelQ-1}, then using \eqref{yfg-yLLL}, and then
\eqref{Yagav-g45}, for each $s\in\{1,\dots,n\}$, $\alpha,\gamma\in\{1,\dots,M\}$,
for every $x'\in{\mathbb{R}}^{n-1}$ and $t=x_n>0$ we obtain
\begin{align}\label{efgug}
2a^{\beta\alpha}_{rs}\Big(Q^{(r)}_{\gamma\beta}\Big)_t(x')
=&\, 2a^{\beta\alpha}_{rs}t^{1-n}(\partial_r E_{\gamma\beta})(x'/t,1)
\nonumber\\[4pt]
=&\,2t^{1-n}(x'/t,1)_s \,k_{\gamma\alpha}(x'/t,1)
\nonumber\\[4pt]
= &\,(x'/t,1)_s\,\big(P_{\gamma\alpha}\big)_t(x')
=\frac{x_s}{t}\big(P_{\gamma\alpha}\big)_t(x').
\end{align}
This takes care of \eqref{kernelQ-3}. The statement in $(b)$ is obtained from
\eqref{ioaTga} and \eqref{kernelQ-3} by writing for every $x'\in{\mathbb{R}}^{n-1}$
and $t>0$
\begin{equation}\label{efgug-2}
\partial_t\Big[\big(P_{\gamma\alpha}\big)_t(x')\Big]
=-\sum_{s=1}^{n-1}\partial_s\Big[\frac{x_s}{t}\big(P_{\gamma\alpha}\big)_t(x')\Big]
=-2\sum\limits_{s=1}^{n-1}a^{\beta\alpha}_{rs}\partial_s
\Big[\big(Q^{(r)}_{\gamma\beta}\big)_t(x')\Big].
\end{equation}
The next task is to prove \eqref{kernelQ-4BB}.
Recalling \eqref{kernelQ-1}, the term in the left hand-side of \eqref{kernelQ-4BB}
evaluated at an arbitrary point $x'\in{\mathbb{R}}^{n-1}$ becomes
\begin{align}\label{efgug-6}
Q^{(n)}_{\gamma\cdot} (x')
&\,= (\partial_t E_{\gamma\cdot})(x',1)
=\big[\partial_t E_{\gamma\cdot}(x',t)\big]\Big|_{t=1}
\nonumber\\[4pt]
&\,=-\sum_{s=1}^{n-1}\big(A_{nn}^{\top}\big)^{-1}\,
\Big[a_{sn}^{\beta\cdot} (\partial_s E_{\gamma\beta})(x',1)\Big]
+\big(A_{nn}^{\top}\big)^{-1}
\big[D_{A^{\!\top}}E_{\gamma\cdot}(x',t)\big]\Big|_{t=1}
\nonumber\\[4pt]
&\,=-\sum_{s=1}^{n-1}\big(A_{nn}^{\top}\big)^{-1}\,
\Big[a_{sn}^{\beta\cdot}Q^{(s)}_{\gamma\beta}(x')\Big]
+\big(A_{nn}^{\top}\big)^{-1}
\Big[a^{\beta\cdot}_{jn} Q^{(j)}_{\gamma\beta}(x')\Big]
\nonumber\\[4pt]
&\,=-\sum_{s=1}^{n-1}\big(A_{nn}^{\top}\big)^{-1}\,
\Big[a_{sn}^{\beta\cdot}Q^{(s)}_{\gamma\beta}(x')\Big]
+\tfrac{1}{2}\big(A_{nn}^{\top}\big)^{-1}\big[P_{\gamma\cdot}(x')\big].
\end{align}
The third equality in \eqref{efgug-6} uses the decomposition of $\partial_t$
as in Lemma \ref{decomp-pt} and \eqref{Tang-partT-2}, the forth equality is based
on \eqref{kernelQ-1} and \eqref{Tang-partT-2}, while the last equality is a consequence
of \eqref{kernelQ-3} specialized to the case when $s=n$.
\end{proof}

It is useful to rephrase the kernel identities from Proposition~\ref{kernelsQj}
in terms of their associated convolution operators. Before doing so, the reader
is advised to recall the piece of notation introduced in \eqref{eq2.13ttt.XX}.

\begin{proposition}\label{kernelsQj-3}
Let $L$ be an operator as in \eqref{L-def}-\eqref{L-ell.X} with the property that
${\mathfrak{A}}_L^{dis}\not=\emptyset$. Denote by $P$ the Poisson kernel for
$L$ from Theorem~\ref{ya-T4-fav}, and fix some coefficient tensor
\begin{equation}\label{ytar-yR5-XX}
A=\bigl(a_{rs}^{\alpha\beta}\bigr)
_{\substack{1\leq r,s\leq n\\ 1\leq\alpha,\beta\leq M}}\in{\mathfrak{A}}_L^{dis}.
\end{equation}
Consider the family of auxiliary kernels
$\big\{Q^{(j)}_{\alpha\beta}\big\}_{j,\alpha,\beta}$
introduced in Definition~\ref{kernelsQj-D} and let $p\in(1,\infty)$.
Then, for every $t>0$, the following identities hold:
\begin{list}{$(\theenumi)$}{\usecounter{enumi}\leftmargin=.8cm
\labelwidth=.8cm\itemsep=0.2cm\topsep=.1cm
\renewcommand{\theenumi}{\alph{enumi}}}
\item for every $f=(f_\alpha)_\alpha\in L^p({\mathbb{R}}^{n-1})$ one has
\begin{equation}\label{kernelQ-4BB-Tg.3}
Q^{(n)}_t\ast f=\tfrac{1}{2}\,P_t\ast A_{nn}^{-1}f
-\sum_{s=1}^{n-1} Q^{(s)}_t\ast A_{sn}A_{nn}^{-1}f
\qquad\mbox{in }\,\,{\mathbb{R}}^{n-1};
\end{equation}
\item if $f=(f_\alpha)_\alpha\in L^p_1({\mathbb{R}}^{n-1})$, then for
each $\gamma\in\{1,\dots,M\}$,
\begin{equation}\label{kernelQ-5}
\partial_t\Big[(P_t\ast f)_\gamma\Big]
=-2\sum_{s=1}^{n-1}a^{\beta\alpha}_{rs}\Big(\big(Q^{(r)}_{\gamma\beta}\big)_t\ast
\partial_s f_\alpha\Big)\qquad\mbox{in }\,\,{\mathbb{R}}^{n-1},
\end{equation}
and for every $r\in\{1,\dots,n-1\}$,
\begin{equation}\label{kernelQ-5.XZ}
\partial_t\Big[\big(Q^{(r)}_t\ast f\big)_\gamma\Big]
=\big(Q^{(n)}_t\ast (\partial_r f)\big)_\gamma
\qquad\mbox{in }\,\,{\mathbb{R}}^{n-1}.
\end{equation}
\end{list}
\end{proposition}

\begin{proof}
Fix $f=(f_\alpha)_\alpha\in L^p({\mathbb{R}}^{n-1})$ and $\gamma\in\{1,...,M\}$.
To obtain \eqref{kernelQ-4BB-Tg.3},  we convolve \eqref{kernelQ-4BB} with $f$
in order to write
\begin{align}\label{kernelQ-4BB-Tg.1}
\Big(Q^{(n)}_t\ast f\Big)_\gamma
&=\big(Q^{(n)}_{\gamma\alpha}\big)_t\ast f_\alpha
\nonumber\\[4pt]
&=\tfrac{1}{2}\Big(\big(A_{nn}^{\top}\big)^{-1}\Big)_{\alpha\mu}
\big(P_{\gamma\mu}\big)_t\ast f_\alpha
-\sum_{s=1}^{n-1}\Big(\big(A_{nn}^{\top}\big)^{-1}\Big)_{\alpha\mu}
a_{sn}^{\beta\mu}\big(Q^{(s)}_{\gamma\beta}\big)_t\ast f_\alpha
\nonumber\\[4pt]
&=\tfrac{1}{2}\big(P_{\gamma\mu}\big)_t\ast\big(A_{nn}^{-1}f\big)_\mu
-\sum_{s=1}^{n-1}a_{sn}^{\beta\mu}\big(Q^{(s)}_{\gamma\beta}\big)_t
\ast\big(A_{nn}^{-1}f\big)_\mu
\nonumber\\[4pt]
&=\tfrac{1}{2}\big(P_t\ast A_{nn}^{-1}f\big)_\gamma
-\sum_{s=1}^{n-1}\big(Q^{(s)}_t\ast A_{sn}A_{nn}^{-1}f\big)_\gamma
\,\,\mbox{ in }\,\,\mathbb{R}^{n-1}.
\end{align}

Moving on, suppose that actually $f\in L^p_1({\mathbb{R}}^{n-1})$ and let
$x'\in{\mathbb{R}}^{n-1}$ be arbitrary. Then we have
\begin{align}\label{efgug-5}
\partial_t\Big[(P_t\ast f)_\gamma(x')\Big]
&=\int_{{\mathbb{R}}^{n-1}}\partial_t\Big[(P_{\gamma\mu})_t(x'-y')\Big]f_\mu(y')\,dy'
\\
&= -2\sum\limits_{s=1}^{n-1}a^{\beta\mu}_{rs}
\int_{{\mathbb{R}}^{n-1}}\partial_{x_s}\Big[\big(Q^{(r)}_{\gamma\beta}\big)_t(x'-y')\Big]
f_\mu(y')\,dy'
\nonumber\\
&=-2\sum_{s=1}^{n-1}a^{\beta\mu}_{rs}\Big(\big(Q^{(r)}_{\gamma\beta}\big)_t\ast
\partial_s f_\mu\Big)(x'),
\nonumber
\end{align}
where in the second equality in \eqref{efgug-5} we have employed \eqref{kernelQ-3BB}.
This proves \eqref{kernelQ-5}. We are left with justifying \eqref{kernelQ-5.XZ}.
If $r\in\{1,\dots,n-1\}$, then making use of \eqref{kernelQ-4} with $j=n$ allows
us to write
\begin{align}\label{kernelQ-4CC}
\partial_t\Big[\big(Q^{(r)}_t\ast f\big)_\gamma\Big]
&= \partial_t\Big[\big(Q^{(r)}_{\gamma\alpha}\big)_t\ast f_\alpha\Big]
=\partial_r\Big[\big(Q^{(n)}_{\gamma\alpha}\big)_t\ast f_\alpha\Big]
\\[4pt]
&=\big(Q^{(n)}_{\gamma\alpha}\big)_t\ast (\partial_r f_\alpha)
=\big(Q^{(n)}_t\ast (\partial_r f)\big)_\gamma\quad\mbox{in }\,\,{\mathbb{R}}^{n-1}.
\end{align}
The proof of the proposition is therefore finished.
\end{proof}

The following convention is designed to facilitate the remaining portion of the
exposition in this section.

\begin{convention}\label{CC-1}
Given two vectors $f$ and $g$, we will use the notation $f\equiv g$ to indicate
that each component of $f$ may be written as a finite linear combination of the
components of $g$. Also, given a coefficient tensor $A=(a^{\alpha\beta}_{jk})_{\alpha,\beta,j,k}$, the notation $M_A f$ is used
to indicate that some (or all) of the components of the vector $f$ are multiplied
with entries from $A$, or from $(A_{nn})^{-1}$. By $\partial_\tau$ we denote any
of the derivatives $\partial_1,...,\partial_{n-1}$, and write $\partial_\tau^k$
for its $k$-fold iteration. Finally, concerning the kernels from \eqref{kernelQ-1},
we agree that $Q^{\rm I}$ denotes any $M\times M$ matrix with entries of the form
$Q^{(s)}_{\alpha\beta}$ where $s\in\{1,\dots,n-1\}$
and $\alpha,\beta\in\{1,\dots,M\}$. On the other hand,
$Q^{\rm II}$ denotes any $M\times M$ matrix with entries of the form $Q^{(n)}_{\alpha\beta}$ where $\alpha,\beta\in\{1,\dots,M\}$.
\end{convention}

Convention~\ref{CC-1} may now be used to succinctly summarize the
identities in Proposition~\ref{kernelsQj-3}, as follows.

\begin{proposition}\label{kernelsQj-Succ}
Retain the hypotheses from Proposition~\ref{kernelsQj-3}. Then
the properties listed below (formulated using Convention~\ref{CC-1})
are true for every $t>0$.
\begin{list}{$(\theenumi)$}{\usecounter{enumi}\leftmargin=.8cm
\labelwidth=.8cm\itemsep=0.2cm\topsep=.1cm
\renewcommand{\theenumi}{\alph{enumi}}}
\item If $f\in L^p({\mathbb{R}}^{n-1})$, then
\begin{equation}\label{kernelQ-4BB-Tg.3Y}
Q^{\rm II}_t\ast f\equiv P_t\ast M_Af+Q^{\rm I}_t\ast M_Af
\qquad\mbox{in }\,\,{\mathbb{R}}^{n-1}.
\end{equation}
\item If $f\in L^p_1({\mathbb{R}}^{n-1})$, then pointwise in
${\mathbb{R}}^{n-1}$ one has
\begin{align}\label{kernelQ-5Ya}
\partial_t\big[Q^{\rm I}_t\ast f\big]
&\equiv Q^{\rm II}_t\ast \partial_\tau f
\equiv P_t\ast(M_A\partial_\tau f)+Q^{\rm I}_t\ast(M_A\partial_\tau f)
\\[6pt]
\label{kernelQ-5Yc}\partial_t\big[P_t\ast f\big]
& \equiv M_A Q^{\rm I}_t\ast\partial_\tau f+M_A Q^{\rm II}_t\ast\partial_\tau f
\\[6pt]
&\equiv M_A Q^{\rm I}_t\ast(M_A\partial_\tau f)+M_A P_t\ast(M_A\partial_\tau f).
\nonumber
\end{align}
\end{list}
\end{proposition}

\begin{proof}
Identity \eqref{kernelQ-4BB-Tg.3Y} is a condensed version of \eqref{kernelQ-4BB-Tg.3}.
The first part in \eqref{kernelQ-5Ya} is a rewriting of \eqref{kernelQ-5.XZ},
while the second part is a consequence of \eqref{kernelQ-4BB-Tg.3Y}.
The first part in \eqref{kernelQ-5Yc} abbreviates \eqref{kernelQ-5},
while the last part follows from the first part and \eqref{kernelQ-4BB-Tg.3Y}.
\end{proof}

We are now in a position to formulate our main identities pertaining to higher
order derivatives of the operator of convolution with the Poisson kernel
under the assumption that the differential operator $L$ has a distinguished
coefficient tensor.

\begin{proposition}\label{KfhI+G}
Let $L$ be an operator as in \eqref{L-def}-\eqref{L-ell.X} with the property that
${\mathfrak{A}}_L^{dis}\not=\emptyset$. Fix some coefficient tensor
$A\in{\mathfrak{A}}_L^{dis}$ and denote by $P$ the Poisson kernel for
$L$ from Theorem~\ref{ya-T4-fav}. Also, let $p\in(1,\infty)$, $k\in{\mathbb{N}}_0$,
and for some $f\in L^p_k({\mathbb{R}}^{n-1})$ define the function
\begin{equation}\label{VfX-2}
u(x',t):=(P_t\ast f)(x'),\qquad\forall\,(x',t)\in{\mathbb{R}}^n_{+}.
\end{equation}
Then, for every  $(x',t)\in{\mathbb{R}}^n_{+}$ the following identity (formulated using Convention~\ref{CC-1})
holds:
\begin{equation}\label{VfX-2:a}
\nabla^k u(x',t)\equiv M_A \big(P_t\ast(M_A\partial_\tau^k f)\big)(x')
+M_A\big(Q^{\rm I}_t\ast (M_A\partial_\tau^k f)\big)(x').
\end{equation}
\end{proposition}

\begin{proof}
Identity \eqref{VfX-2:a} follows by induction on $k$ from identities
\eqref{kernelQ-5Ya}, \eqref{kernelQ-5Yc} and the fact that for each
$\ell\in{\mathbb{N}}$ and each $t>0$, we have
\begin{equation}\label{VfX-2BBB}
\partial_\tau^\ell\big(P_t\ast g\big)=P_t\ast \partial_\tau^\ell g
\quad\mbox{ and }\quad
\partial_\tau^\ell\big(Q^{\rm I}_t\ast g\big)=Q^{\rm I}_t\ast \partial_\tau^\ell g
\quad\mbox{in }\,\,{\mathbb{R}}^{n-1},
\end{equation}
for every $g\in L^p_\ell({\mathbb{R}}^{n-1})$.
\end{proof}

All the ingredients are now in place to proceed with the proof our main result.

\begin{proof}[Proof of Theorem \ref{them:Dir-l}]
Fix $p\in(1,\infty)$, $\ell\in{\mathbb{N}}_0$, and $f\in L^p_\ell(\mathbb{R}^{n-1})$.
The fact that $u$ defined as in \eqref{eqn-Dir-l:u} satisfies the first and
last conditions in \eqref{Dir-BVP-pLL:l} is a consequence of \eqref{uahgab-UBVCX}
and \eqref{smetg}. In addition, uniqueness for \eqref{Dir-BVP-pLL:l} is a
consequence of Theorem~\ref{YF-TY.6yh}. Finally, from \eqref{VfX-2:a}, \eqref{GTdS-2},
and the estimate in \eqref{smetg}, we deduce that the function \eqref{eqn-Dir-l:u}
also satisfies \eqref{Dir-BVP-p2F.ii:l}.
\end{proof}

\section{Examples of boundary problems of mathematical physics}
\setcounter{equation}{0}
\label{S-5}

In this section we present some examples involving differential operators
of basic importance in mathematical physics. For a more detailed discussion
(as well as a broader perspective) in this regard, the interested reader
is referred to \cite{MaMiMiMi}.

\subsection{Scalar second order elliptic equations}\label{SSe-5.1}
Assume that the $n\times n$ matrix $A=(a_{rs})_{r,s}\in{\mathbb{C}}^{n\times n}$
with complex entries satisfies the ellipticity condition
\begin{equation}\label{YUjhv-753}
\inf_{\xi\in S^{n-1}}{\rm Re}\,\bigl[a_{rs}\xi_r\xi_s\bigr]>0,
\end{equation}
and consider the elliptic differential operator $L={\rm div}A\nabla$
in ${\mathbb{R}}^n_{+}$. From Remark~\ref{Tavav-yag8} we know that
${\mathfrak{A}}_L^{dis}\not=\emptyset$ and, in fact, $A_{\rm sym}\in {\mathfrak{A}}_L^{dis}$. Keeping this in mind, Proposition~\ref{uniq:double->!Poisson},
\eqref{YTcxb-ytSH}, and \eqref{yfg-yLLL}, eventually give that
\begin{equation}\label{Uahab8a-hab}
P(x'):=\frac{2}{\omega_{n-1}\sqrt{{\rm det}\,(A_{\rm sym})}}
\frac{1}{\big\langle(A_{\rm sym})^{-1}(x',1),(x',1)\big\rangle^{\frac{n}{2}}},
\qquad\forall\,x'\in{\mathbb{R}}^{n-1},
\end{equation}
is the (unique, by Proposition~\ref{uniqueness-Poisson}) Poisson kernel
for the operator $L={\rm div}A\nabla$ in ${\mathbb{R}}^n_{+}$. It is reassuring to
observe that \eqref{Uahab8a-hab} reduces precisely to \eqref{Uah-TTT} in the case
when $A=I$ (i.e., when $L$ is the Laplacian).

Going further, by invoking Theorem~\ref{them:Dir-l} we obtain that
for each $\ell\in\mathbb{N}_0$
the $\ell$-th order Dirichlet boundary value problem \eqref{Dir-BVP-pLL:l} is
well-posed when $L={\rm div}A\nabla$. Moreover, the solution $u$ satisfies
\eqref{Dir-BVP-p2F.ii:l}, and is given at each point
$(x',t)\in{\mathbb{R}}^n_{+}$ by the formula
\begin{equation}\label{kanb-uTn}
u(x',t)=\frac{2t}{\omega_{n-1}\sqrt{{\rm det}\,(A_{\rm sym})}}
\int_{{\mathbb{R}}^{n-1}}\frac{f(y')}{\big\langle(A_{\rm sym})^{-1}(x'-y',t),(x'-y',t)
\big\rangle^{\frac{n}{2}}}\,dy'.
\end{equation}

\subsection{The case of the Lam\'e system of elasticity}\label{SSe-5.2}
Recall that the Lam\'e operator in ${\mathbb{R}}^n$ has the form
\begin{equation}\label{TYd-YG-76g}
Lu:=\mu\Delta u+(\lambda+\mu)\nabla{\rm div}\,u,\qquad u=(u_1,...,u_n)\in{\mathscr{C}}^2,
\end{equation}
where the constants $\lambda,\mu\in{\mathbb{R}}$ (typically called Lam\'e moduli),
are assumed to satisfy
\begin{equation}\label{Yfhv-8yg}
\mu>0\,\,\mbox{ and }\,\,2\mu+\lambda>0.
\end{equation}
Condition \eqref{Yfhv-8yg} is equivalent to the demand that the Lam\'e system \eqref{TYd-YG-76g} satisfies the Legendre-Hadamard ellipticity condition \eqref{L-ell.X}.
To illustrate the manner in which the Lam\'e system \eqref{TYd-YG-76g}
may be written in infinitely many ways as in \eqref{L-def}, for each
$\theta\in{\mathbb{R}}$ introduce
\begin{equation}\label{Lame-1}
a_{rs}^{\alpha\beta}(\theta):=\mu\,\delta_{rs}\delta_{\alpha\beta}
+(\lambda+\mu-\theta)\,\delta_{r\alpha}\delta_{s\beta}
+\theta\,\delta_{r\beta}\delta_{s\alpha},\qquad 1\leq\alpha,\beta,r,s\leq n.
\end{equation}
Then for each $\theta\in\mathbb{R}$ one can show that the Lam\'e
operator \eqref{TYd-YG-76g} may be regarded as having the form \eqref{L-def}
for the coefficient tensor $A=A(\theta):=\bigl(a_{rs}^{\alpha\beta}(\theta)\bigr)
_{\substack{1\leq r,s\leq n\\ 1\leq\alpha,\beta\leq n}}$ with entries
as in \eqref{Lame-1}. In short, $A(\theta)\in{\mathfrak{A}}_L$ for each
$\theta\in{\mathbb{R}}$.

Regarding the existence of a value for the parameter $\theta\in{\mathbb{R}}$ which makes
$A(\theta)$ a distinguished coefficient tensor for the Lam\'e system, we note the
following result.

\begin{lemma}\cite{MaMiMi},\cite{MaMiMiMi}\label{gEJ}
One has $A(\theta)\in{\mathfrak{A}}_L^{dis}$ if and only if
$\theta=\frac{\mu(\lambda+\mu)}{3\mu+\lambda}$. Moreover, corresponding to
this value of $\theta$, the entries in $A(\theta)$ become
for $\alpha,\beta,r,s\in\{1,\dots,n\}$
\begin{equation}\label{BUIg-17XX}
a^{\alpha\beta}_{rs}
=\mu\delta_{rs}\delta_{\alpha\beta}
+\frac{(\lambda+\mu)(2\mu+\lambda)}{3\mu+\lambda}\delta_{r\alpha}\delta_{s\beta}
+\frac{\mu(\lambda+\mu)}{3\mu+\lambda}\delta_{r\beta}\delta_{s\alpha}.
\end{equation}
\end{lemma}

In turn, for the choice of coefficient tensor as in \eqref{BUIg-17XX},
a straightforward calculation using the expression of the fundamental solution that can be found in e.g. \cite{DM} proves that \eqref{yfg-yLLL} is satisfied if
we consider, for every $\alpha,\beta\in\{1,...,n\}$,
\begin{equation}\label{Lame-9RR.Pf}
k_{\alpha\beta}(x):=\frac{2\mu}{3\mu+\lambda}
\frac{\delta_{\alpha\beta}}{\omega_{n-1}}\frac{1}{|x|^{n}}
+\frac{\mu+\lambda}{3\mu+\lambda}\frac{n}{\omega_{n-1}}
\frac{x_\alpha x_\beta}{|x|^{n+2}},\qquad x\in{\mathbb{R}}^{n}\setminus\{0\}.
\end{equation}
Based on this and \eqref{Yagav-g45},
we obtain that the unique Poisson kernel for the Lam\'e system \eqref{TYd-YG-76g}
is the matrix-valued function $P=(P_{\alpha\beta})_{1\leq\alpha,\beta\leq n}:
{\mathbb{R}}^{n-1}\to{\mathbb{R}}^{n\times n}$ whose entries are given
for each $\alpha,\beta\in\{1,...,n\}$ and $x'\in{\mathbb{R}}^{n-1}$ by
\begin{equation}\label{Lame-9RR.L}
P_{\alpha\beta}(x')=\frac{4\mu}{3\mu+\lambda}
\frac{\delta_{\alpha\beta}}{\omega_{n-1}}\frac{1}{(|x'|^2+1)^{\frac{n}{2}}}
+\frac{\mu+\lambda}{3\mu+\lambda}\frac{2n}{\omega_{n-1}}
\frac{(x',1)_\alpha(x',1)_\beta}{(|x'|^2+1)^{\frac{n+2}{2}}},
\end{equation}
In concert with Theorem~\ref{them:Dir-l}, this analysis allows us to formulate
the following well-posedness result for the $\ell$-th order Dirichlet problem
for the Lam\'e system in the upper-half space.

\begin{theorem}\label{tuig-333}
Assume that the Lam\'e moduli $\lambda,\mu$ satisfy \eqref{Yfhv-8yg}. Then for every
$p\in(1,\infty)$, and for each $\ell\in\mathbb{N}_0$, the $\ell$-th order
Dirichlet boundary value problem \eqref{Dir-BVP-pLL:l} is well-posed for
the Lam\'e system \eqref{TYd-YG-76g}. In addition, the solution
$u=(u_\alpha)_{1\leq\alpha\leq n}$ corresponding to the boundary datum
$f=(f_\beta)_{1\leq\beta\leq n}\in L^p_\ell({\mathbb{R}}^{n-1})$ is given by
\begin{align}\label{kanb-u5yuh}
u_\alpha(x',t)=&\,\frac{4\mu}{3\mu+\lambda}
\frac{1}{\omega_{n-1}}\int_{{\mathbb{R}}^{n-1}}
\frac{t}{(|x'-y'|^2+t^2)^{\frac{n}{2}}}\,f_\alpha(y')\,dy'
\nonumber\\[4pt]
&\,+\frac{\mu+\lambda}{3\mu+\lambda}\frac{2n}{\omega_{n-1}}\int_{{\mathbb{R}}^{n-1}}
\frac{t\,(x'-y',t)_\alpha(x'-y',t)_\beta}{(|x'-y'|^2+t^2)^{\frac{n+2}{2}}}\,
f_\beta(y')\,dy',
\end{align}
at each point $(x',t)\in{\mathbb{R}}^n_{+}$, and satisfies \eqref{Dir-BVP-p2F.ii:l}.
\end{theorem}

\end{document}